\documentclass{amsart}
\usepackage{amssymb,hyperref}
\newtheorem{theorem}{Theorem}[section]
\newtheorem{lemma}[theorem]{Lemma}
\newtheorem{cor}[theorem]{Corollary}
\newcommand{\Sym}{\mathop{\textrm{Sym}}}
\renewcommand{\wr}{\mathop{\textrm{wr}}}
\newcommand{\CI}{\mathop{\mathrm{DCI}}}
\newcommand{\Cay}{\mathop{\mathrm{Cay}}}
\newcommand{\Aut}{\mathop{\mathrm{Aut}}}

\def\Z{\mathbb Z}
\begin{document}
\title
{Elementary proof that $\Z_p^4$ is a DCI-group}

\author[J. Morris]{Joy Morris}
\address{Joy Morris, Department of Mathematics and Computer Science,
University of Lethbridge,
Lethbridge, AB. T1K 3M4. Canada}
\email{joy@cs.uleth.ca}

\thanks{This research was supported in part by the National Science
 and Engineering Research Council of Canada.}

\begin{abstract}
A finite group $R$ is a $\CI$-group if, whenever $S$ and $T$ are subsets of $R$ with the Cayley graphs $\Cay(R,S)$ and $\Cay(R,T)$ isomorphic, there exists an automorphism $\varphi$ of $R$ with $S^\varphi=T$.

Elementary abelian groups of order $p^4$ or smaller are known to be $\CI$-groups, while those of sufficiently large rank are known not to be $\CI$-groups.
The only published proof that elementary abelian groups of order $p^4$ are $\CI$-groups uses Schur rings and does not work for $p=2$ (which has been separately proven using computers). This paper provides a simpler proof that works for all primes. Some of the results in this paper also apply to elementary abelian groups of higher rank, so may be useful for completing our determination of which elementary abelian groups are $\CI$-groups.
\end{abstract}

\maketitle
\section{Introduction}\label{sec:intro}
The classification of  $\CI$-groups is an open problem in the theory of Cayley graphs and is closely related to the isomorphism problem for graphs. It is a long-standing problem that has been worked on a lot, see \cite{DMV,Lisurvey} for additional background. The formulation of this problem was introduced by Babai in \cite{babai}. Elementary abelian groups of order $p^4$ or smaller are known to be $\CI$-groups \cite{Turner,Godsil,BrianLewis,Da,HirasakaMuzychuk}, while those of sufficiently large rank are known not to be $\CI$-groups \cite{Muz,Spiga,Somlai}.
The only published proof that elementary abelian groups of order $p^4$ are $\CI$-groups \cite{HirasakaMuzychuk}, uses Schur rings and does not work for $p=2$ (which has been separately proven using computers). This paper provides a simpler proof that works for all primes. It is based on work from the author's PhD thesis \cite{thesis} (which was completed concurrently with the Hirasaka-Muzychuk result), but has been considerably shortened and simplified. Some of the results in this paper have been newly generalised to apply to elementary abelian groups of higher rank, so may be useful for completing our determination of which elementary abelian groups are $\CI$-groups.

Let $R$ be a finite group and let $S$ be a subset of $R$. The \textit{Cayley digraph} of $R$ with connection set $S$, denoted $\Cay(R,S)$, is the digraph with vertex set $R$ and with $(x,y)$ being an arc if and only if $x^{-1}y\in S$. Now, $\Cay(R,S)$ is said to have the \textit{Cayley isomorphism} property for digraphs, or be a $\CI$-\textit{graph} for short, if whenever $\Cay(R,S)$ is isomorphic to $\Cay(R,T)$, there exists an automorphism $\varphi$ of $R$ with $\varphi(S)=T$. Clearly, $\Cay(R,S)\cong \Cay(R,\varphi(S))$ for every $\varphi\in \Aut(R)$ so that for a $\CI$-graph, solving the isomorphism problem boils down to understanding the automorphisms of the group $R$.
 The group $R$ is a $\CI$-\textit{group} if $\Cay(R,S)$ is a $\CI$-graph for every subset $S$ of $R$. Moreover, $R$ is a $\mathrm{CI}$-group if $\Cay(R,S)$ is a $\CI$-graph for every inverse-closed subset $S$ of $R$. Thus every $\CI$-group is a $\mathrm{CI}$-group.

Throughout this paper, $p$ will always denote a prime number, and calculations are always performed modulo $p$ (i.e., in $\Z_p$).

\begin{theorem}\label{thrm:main}
Let $p$ be a prime number and let $R$ be the elementary abelian group of order $p^4$. Then $R$ is a $\CI$-group.
\end{theorem}

The structure of the paper is straightforward. In Section~\ref{basictools}, we provide some preliminary definitions and notation, and reproduce some lemmas from other papers that will apply directly to our situation, including our main tool. In Sections~\ref{tau3},~\ref{tau2tau2'nice} and~\ref{therest} we complete the proof of Theorem~\ref{thrm:main}.   Where possible, we will state our results in the more general context of Cayley graphs on arbitrary elementary abelian groups, as some of the results may be useful for proving that elementary abelian groups of higher rank are $\CI$-groups.

\section{Preliminary results and notation}\label{basictools}
Babai~\cite{babai} proved a very useful criterion for determining when a finite group $R$ is a $\CI$-group and, more generally, when $\Cay(R,S)$ is a $\CI$-graph.

\begin{lemma}\label{lemma}Let $R$ be a finite group and let $S$ be a subset of $R$. Then $\Cay(R,S)$ is a $\CI$-graph if and only if $\Aut(\Cay(R,S))$ contains a unique conjugacy class of regular subgroups isomorphic to $R$.
\end{lemma}

Let $\Omega$ be a finite set and let $G$ be a permutation group on $\Omega$. The $2$-\textit{closure} of $G$, denoted $G^{(2)}$, is the set $$\{\pi\in \Sym(\Omega)\mid \forall (\omega,\omega')\in \Omega^2, \textrm{there exists } g_{\omega\omega'}\in G \textrm{ with }\pi((\omega,\omega'))=g_{\omega\omega'}((\omega,\omega'))\},$$
where $\Sym(\Omega)$ is the symmetric group on $\Omega$. Observe that in the definition of $G^{(2)}$, the element $g_{\omega\omega'}$ of $G$ may depend upon the ordered pair $(\omega,\omega')$. The group $G$ is said to be $2$-\textit{closed} if $G=G^{(2)}$.

It is easy to verify that $G^{(2)}$ is a subgroup of $\Sym(\Omega)$ containing $G$ and, in fact, $G^{(2)}$ is the smallest (with respect to inclusion) subgroup of $\Sym(\Omega)$ preserving every orbital digraph of $G$. It follows that the automorphism group of a graph is $2$-closed.
Therefore Lemma~\ref{lemma} immediately yields:

\begin{lemma}[Lemma 2.2 of~\cite{DMV}]\label{babaistrong}
Let $R$ be a finite group and let $R_L$ be the left regular representation of $R$ in $\Sym(R)$. If, for every $\pi\in \Sym(R)$, the groups $R_L$ and $\pi^{-1}R_L\pi$ are conjugate in $\langle R_L,\pi^{-1}R_L\pi\rangle^{(2)}$, then $R$ is a $\CI$-group.
\end{lemma}

We will use this formulation of Babai's criterion without comment in our proof of Theorem~\ref{thrm:main}.

We now set up some notation that will be used throughout the rest of the paper.

Let $R$ be an elementary abelian group of rank $n$. Set $G=\langle R_L, \pi^{-1}R_L\pi\rangle$. Let $P$ be a Sylow $p$-subgroup of $G$ with $R\leq P$ and let $T$ be a Sylow $p$-subgroup of $\Sym(\Omega)$ with $P\leq T$. From Sylow's theorems, replacing $\pi^{-1}R_L\pi$ by a suitable $G$-conjugate, we may assume that $\pi^{-1}R_L\pi\leq P$, so that in fact $G=P$. From now on we will refer exclusively to $G$, but keep in mind that $G$ itself is a $p$-group.

Observe that the group $T$ is $\Z_p\wr \ldots\wr\Z_p$ ($n$ copies of $\Z_p$), which has a unique system of imprimitivity with blocks of size $p^i$ for each $0 \le i \le n$. Since $R_L$ and $\pi^{-1}R_L\pi$ are acting regularly, they must admit these same systems of imprimitivity. For $0 \le i \le n$, let $\mathcal B_i$ be the system of imprimitivity of $T$ that consists of blocks of size $p^i$.

For each $0 \le i \le n-1$, choose $\tau_i$ to be an element of $R_L$ that fixes each set in $\mathcal B_{i+1}$ setwise, and has order $p$ in its action on the sets in $\mathcal B_i$. Notice that $\langle \tau_0, \ldots, \tau_{n-1} \rangle =R_L$. Let $v$ be a fixed element of $R$ (recall that both $R_L$ and $\pi^{-1}R_L\pi$ are acting on $R$). For each $0 \le i \le n-1$, define $\tau_i'$ to be the unique element of $\pi^{-1}R_L\pi$ such that $\tau_i'(v)=\tau_i(v)$. 

If we say that a permutation fixes every element of $\mathcal B_i$ for some $i$, this means that the blocks of $\mathcal B_i$ are all fixed setwise, and does not imply that any point of $R$ is fixed.

For any $v \in R$, use $B_v$ to denote the element of $\mathcal B_1$ that contains $v$, and $C_v$ to denote the element of $\mathcal B_2$ that contains $v$. In some cases, we will be dealing with two or even three systems of imprimitivity of $G$ with blocks of size $p$; in this event, we call the additional systems $\mathcal B_1'$ and $\mathcal B_1''$, and $B_v'$ and $B_v''$ will denote the elements of these systems (respectively) that contain $v$.

The following result is a restatement of Proposition 2.3 of~\cite{Zpn}.

\begin{lemma}\label{samePgp}
Let $\pi^{-1}R_L\pi$ and $R_L$ be isomorphic to $\Z_p^n$ and lie in the same Sylow $p$-subgroup of $\Sym(R)$.  Suppose that $\tau \in R_L$ fixes every element of $\mathcal B_i$ for some $i\ge 1$, and has order $p$ on the elements of $\mathcal B_{i-1}$. Let $\tau' \in \pi^{-1}R_L\pi$ be such that $\tau'(B)=\tau(B)$ for some $B \in \mathcal B_{i-1}$. Then $\tau'(B')=\tau(B')$ for every $B' \in \mathcal B_{i-1}$.
\end{lemma}

The hypothesis that $\pi^{-1}R_L\pi$ and $R_L$ lie in the same Sylow $p$-subgroup of $\Sym(R)$ will generally be considered to be part of the notation we have established (that $G$ is a $p$-group), so will be tacitly assumed in  our results. Since it is the key assumption needed to prove the above result, however, we have stated it explicitly this once.

We introduce a bit more notation that will be required for the next result, and will be used in Section~\ref{tau3}. Let $K$ be the kernel of the action of $G$ on $\mathcal{B}_1$.
We define an equivalence relation $\equiv$ on $\Omega$. Given $x,x' \in R$, we have $x\equiv x'$ whenever, for every $\rho\in G$, $\rho\vert_{B_x} = {\rm id}\vert_{B_x}$ if and only if $\rho\vert_{B_{x'}} = {\rm id}\vert_{B_{x'}}$ (or equivalently, $\rho\vert_{B_x}$ is a $p$-cycle if and only if $\rho\vert_{B_{x'}}$ is a $p$-cycle). Let $\mathcal E$ denote the set of equivalence classes of $\equiv$.

\begin{lemma}\label{tedslemma}
For every $\rho\in K$ and for every $E\in \mathcal{E}$, the permutation $\rho_E:R\to R$, fixing $R\setminus E$ pointwise and acting on $E$ as $\rho$ does, lies in  $G^{(2)}$.
\end{lemma}
\begin{proof}
This is Lemma~$2$ in~\cite{Dobson1995}.  (We remark that~\cite[Lemma 2]{Dobson1995} is only stated for graphs, but the result holds for each orbital digraph of $G$, and hence for $G^{(2)}$.)
\end{proof}

The final result that we require that we require from the existing literature is Proposition 2.7 from~\cite{Zpn}, restated slightly. 

\begin{lemma}\label{tau1}
Let $R$ be an elementary abelian group of rank $n$.
Under the notation we have established, if $\tau_0,\ldots,\tau_{n-2} \in \pi^{-1}R_L\pi$, then $R_L$ and $\pi^{-1}R_L\pi$ are conjugate in $G^{(2)}$.
\end{lemma}

\section{$\tau_1'$}\label{tau3}
In this section, we prove a key lemma that will allow us to assume that for any elementary abelian group $R$ (regardless of the rank), the centre of $G$ has order at least $p^2$. Specifically, we
will prove that there exists $\psi \in G^{(2)}$ such that $\psi$ commutes with $\tau_0$, and $\psi\pi^{-1}R_L\pi\psi$ contains $\tau_1$. Notice that Lemma~\ref{samePgp} immediately implies that $\tau_0'=\tau_0 \in Z(G)$, so proving this will allow us to assume that $|Z(G)|\ge p^2$; specifically, that $\tau_0=\tau_0'$ and $\tau_1'=\tau_1$ in the final section.  

Since the following lemma applies quite broadly, we state it in general terms. We will follow with a corollary that more clearly applies this lemma.

\begin{lemma}\label{fixing-blocks}
Let $R$ be an elementary abelian group of rank $n$. Under the notation we have established, if $\tau' \in \pi^{-1}R_L\pi$ and $\tau \in R_L$ such that $\tau^{-1}\tau'$ fixes every element of $\mathcal B_1$, then there exists $\psi \in G^{(2)}$ such that $\psi^{-1}\pi^{-1}R_L\pi\psi$ contains $\tau$. 

Furthermore, if $\alpha$ fixes every element of $\mathcal E$, then $\psi$ commutes with $\alpha$.
\end{lemma}

\begin{proof}
By Lemma~\ref{samePgp}, we may assume that $\tau_0'=\tau_0$.  Replace $\tau'$ if necessary by $\tau_0^c\tau'$ for some $c$, so that $\tau^{-1}\tau'(v)=v$. The hypotheses of this lemma are still satisfied.

We will use $g$ to denote $\tau^{-1}\tau'$. Notice that for any $B \in \mathcal B_1$, since $g$ fixes $B$ setwise and is in the $p$-group $G$, we must have $g\vert_B=\tau_0^{c_B}\vert_B$ for some $c_B$.
 Furthermore, if $B, B' \in \mathcal B_1$ and $B, B' \subseteq E \in \mathcal E$, then using $\rho=\tau_0^{-c_B}g$ in the definition of $\equiv$, we see that we must have $c_B=c_{B'}$.  If $c_B=c_{B'}$ for every $B, B'\in \mathcal B_1$, then since $g(v)=v$, we have $c_B=0$ for every $B \in \mathcal B_1$, and hence $\tau'=\tau$, so letting $\psi={\rm id}$ yields the desired conclusion.  Therefore, in the remainder of this proof, we may assume that $|\mathcal E|>1$, and that there exists $B \in \mathcal B_1$ such that $c_B\neq 0$.

Now we show that for any $E \in \mathcal E$, $\tau(E)\neq E$. Let $\sigma \in R_L$ be such that $\sigma(v) \in B$, where $c_B \neq 0$, and let $\sigma' \in \pi^{-1}R_L \pi$ be such that $\sigma'(v)=\sigma(v)$, so $\sigma^{-1}\sigma'(v)=v$. Let $w=\tau(v)$. Then since both $R_L$ and $\pi^{-1}R_L \pi$ are abelian, we have
\begin{eqnarray*}
\sigma^{-1}\sigma'(w)&=&\sigma^{-1}\tau'\sigma'(v)\\
&=&\sigma^{-1}\tau_0^{c_B}\tau\sigma(v)\\
&=&\tau_0^{c_B}(w).
\end{eqnarray*}
(To get the second line, we are using the definition of $c_B$.) Thus using $\rho=\sigma^{-1}\sigma'$ in the definition of $\equiv$, we see that when $v \in E \in \mathcal E$, we have $w=\tau(v)\not\equiv v$. Since $\mathcal E$ is invariant under $G$ and $G$ is transitive, this proves that for any $E \in \mathcal E$, $\tau(E) \neq E$.

We know that $\mathcal E$ consists of $p^j$ classes, for some $j \ge 1$.
Since $R_L$ is elementary abelian, $\tau$ has order $p$, so the classes of $\mathcal E$ can be 
partitioned into $p^{j-1}$ orbits of $\tau$, which we will refer to as orbit 1, \ldots, orbit $p^{j-1}$. For 
orbit $i$, we arbitrarily choose one element of $\mathcal E$ in that orbit, and label it $E_i$. Now, the elements of $\mathcal E$ are 
$$\{\tau^k(E_i): 0 \le k \le p-1, 1 \le i \le p^{j-1}\}.$$
For $1 \le i \le p^{j-1}$, let $B_i \in \mathcal B_1$ be such that $B_i \subseteq E_i$. Define $\psi$ as follows: for any $i$, $\psi\vert_{E_i}={\rm id}\vert_{E_i}$, and for any integer $k$, $$\psi\vert_{\tau^k(E_i)}=\tau_0^{\sum_{t=0}^{k-1}c_{\tau^t(B_i)}}\vert_{\tau^k(E_i)}.$$
Notice that since $\tau'$ has order $p$, $\sum_{t=0}^{p-1}c_{\tau^t(B_i)}\equiv 0 \pmod{p}$, making this definition consistent.

Now, $\psi$ is a product of elements of the form $(\tau_0^c)_E$ (using the notation of Lemma~\ref{tedslemma}), so by Lemma~\ref{tedslemma}, we have $\psi \in G^{(2)}$.  Since $\tau_0 \in Z(G)$, it is clear that $\psi$ commutes with any $\alpha$ that fixes every element of $\mathcal E$.

We claim that $\psi^{-1}\tau'\psi=\tau$, which will complete the proof.  Let $x \in R$ be arbitrary, and let $k,i$ be such that $x \in \tau^k(E_i)$. Then since $\tau_0\in Z(G)$, we have 
\begin{eqnarray*}
\psi^{-1}\tau'\psi(x)&=&\psi^{-1}\tau'\tau_0^{\sum_{t=0}^{k-1}c_{\tau^t(B_i)}}(x)\\
&=& \psi^{-1}\tau_0^{\sum_{t=0}^{k-1}c_{\tau^t(B_i)}}\tau'(x)\\
&=&\psi^{-1}\tau_0^{\sum_{t=0}^{k}c_{\tau^t(B_i)}}\tau(x)\\
&=&\tau(x)
\end{eqnarray*}
\end{proof}

\begin{cor}\label{tau3'}
Let $R$ be an elementary abelian group of rank $n$. Under the notation we have established, there exists $\psi \in G^{(2)}$ such that $\psi$ commutes with $\tau_0$, and $\psi^{-1}\pi^{-1}R\pi\psi$ contains $\tau_1$. Thus we may assume that $\tau_0, \tau_1 \in Z(G)$.
\end{cor}

\begin{proof}
By Lemma~\ref{samePgp}, we have $\tau_0'=\tau_0$ fixes every element of $\mathcal E$ setwise by the definition of $\equiv$, and for any $B \in \mathcal B_1$, we have $\tau_1^{-1}\tau_1'(B) =B$. Thus $\tau_1'$ and $\tau_1$ fulfill the requirements of Lemma~\ref{fixing-blocks}, with $\tau_0$ taking the role of $\alpha$.
\end{proof}

\section{An easy case}\label{tau2tau2'nice}

In this section, we will consider the possibility that $\tau_2'$ is ``close" to $\tau_2$ (meaning that $\tau_2^{-1}\tau_2'$ fixes every block of $\mathcal B_1$ or some other system of imprimitivity with blocks of size $p$). We determine some circumstances under which this situation must arise, and conclude that the proof of Theorem~\ref{thrm:main} is complete under these circumstances.

Corollary~\ref{tau3'} has concluded that we may assume $\tau_0, \tau_1 \in Z(G)$. This means that for any $g \in G$ and any $w \in R$, if $g(w)=\tau_0^i\tau_1^j(w) \in C_w$, then $g^k(w)=\tau_0^{ki}\tau_1^{kj}$. Thus $G$ lies in multiple Sylow $p$-subgroups of $\Sym(R)$; in particular, every Sylow $p$-subgroup of $\Sym(R)$ that admits $\mathcal B_2, \ldots, \mathcal B_{n-1}$ as systems of imprimitivity as well as admitting any one of the $p+1$ systems of blocks of size $p$ that are preserved by $R_L$. (In addition to $\mathcal B_1$, these are the orbits of  $(\tau_0)^i\tau_1$ for $0 \le i \le p-1$.) Our argument about the action of $g$ demonstrates that orbits of $G_v$ meet any block of $\mathcal B_2$ in either a single point, one of these blocks of size $p$, or the entire block of $\mathcal B_2$. 

By our observations above, we may replace $\mathcal B_1$ by any of the other systems of imprimitivity with blocks of size $p$ that are refinements of $\mathcal B_2$ and are admitted by $G$, and redefine $\equiv$ and $\mathcal E$ accordingly. This concept gives us the following result.

\begin{lemma}\label{nop^2-wreathing}
Let $R$ be an elementary abelian group. Under the notation we have established, if $\tau_2^{-1}\tau_2'$ fixes every block of some system of imprimitivity of $R_L$ with blocks of size $p$ that is a refinement of $\mathcal B_2$, then there exists $\psi \in G^{(2)}$ that commutes with $\tau_0$ and $\tau_1$ and such that $\tau_2 \in \psi^{-1}\pi^{-1}R_L\pi\psi$. 
\end{lemma}

\begin{proof}
By Corollary~\ref{tau3'}, we may assume that $\tau_0,\tau_1 \in Z(G)$. Replacing $\mathcal B_1$ by the system of imprimitivity whose blocks are fixed by $\tau_2^{-1}\tau_2'$ and applying Lemma~\ref{fixing-blocks}, we see that there exists $\psi \in G^{(2)}$ such that $\tau_2 \in \psi^{-1}\pi^{-1}R_L \pi\psi$.  Furthermore, $\tau_0$ and $\tau_1$ commute with $\tau_2^{-1}\tau_2'$ and so have the property of $\alpha$ in the statement of that lemma. Thus, $\psi$ commutes with both $\tau_0$ and $\tau_1$.
\end{proof}

The next lemma and corollary point out a circumstance under which the above special situation must arise. In order to find appropriate elements of $G^{(2)}$ that will conjugate $\pi^{-1}R_L\pi$ to $R_L$, the orbits of particular subgroups of $G$ will be key. The subgroups of interest are those that fix the vertex $v$, while simultaneously fixing every block of $\mathcal B_i$ for some $i$. For any fixed $i$, we will denote this subgroup by $G_{v,\mathcal B_i}$. 

\begin{lemma}\label{still-no-p^2-wreathing}
Let $R$ be an elementary abelian group of rank 4. Under the notation we have established, suppose there are blocks $C, C' \in \mathcal B_2$ such that $G_{v,\mathcal B_2}$ fixes each $B\in \mathcal B_1$ with $B \subseteq C$, and also fixes each $B' \in \mathcal B_1$ with $B' \subseteq C'$. Further suppose that  $\alpha\in R_L$ is such that $\alpha(C_v)=C$ and there is no $i$ such that $\alpha^i(C_v)=C'$. Then $\tau_2^{-1}\tau_2'$ fixes every element of $\mathcal B_1$ (setwise).
\end{lemma}

\begin{proof}
There are $p^2$ elements of $\mathcal B_2$.  By our assumptions, if $\beta\in R_L$ is such that $\beta(C_v)=C'$, then any element of $\mathcal B_2$ can be written uniquely as $\alpha^i\beta^j(C_v)$, where $0 \le i \le p-1$, and $0 \le j \le p-1$.

We will show that if 
$G_{v,\mathcal B_2}$ fixes each block of $\mathcal B_1$ in $\alpha^i\beta^j(C_v)$ setwise, then $G_{v,\mathcal B_2}$ fixes each block of $\mathcal B_1$ in $\alpha^{i+1}\beta^j(C_v)$ setwise; and likewise, $G_{v,\mathcal B_2}$ fixes each block of $\mathcal B_1$ in $\alpha^{i}\beta^{j+1}(C_v)$ setwise. Inductively, this will show that $G_{v,\mathcal B_2}$ fixes each block of $\mathcal B_1$ setwise, so in particular, $\tau_2^{-1}\tau_2' \in G_v$, which fixes each block of $\mathcal B_2$ setwise, must actually fix each block of $\mathcal B_1$ setwise.

Let $g$ be an arbitrary element of $G_{v, \mathcal B_2}$. Let $\gamma\in \{\alpha,\beta\}$. Since (by assumption) $g$ fixes each block of $\mathcal B_1$ in $\alpha^i\beta^j(C_v)$ setwise, in particular we may assume that $g\alpha^i\beta^j(v)=\tau_0^k\alpha^i\beta^j(v)$ for some $k$. Thus $g'=\alpha^{-i}\beta^{-j}\tau_0^{-k}g\alpha^i\beta^j \in G_v$, and since $g$ and $\tau_0$ fix every element of $\mathcal B_2$, so does $g'$, so $g' \in G_{v,\mathcal B_2}$.  We therefore have that $g'$ fixes every $B \in \mathcal B_1$ with $B \subset \gamma(C_v)$. Let $B$ be arbitrary subject to the constraints $B \in \mathcal B_1$ and $B \subset \gamma\alpha^i\beta^j(C_v)$. Let $B' \subset \gamma(C_v)$ be such that $\alpha^i\beta^j(B')=B$. We know that $g'$ fixes $B'$, so $\alpha^{-i}\beta^{-j}\tau_0^{-k}g\alpha^i\beta^j(B')=B'$. This implies that $\tau_0^k(B)=g(B)$. Since $\tau_0$ fixes $B$ setwise, so must $g$. This completes the proof.
\end{proof}

\begin{cor}\label{stillstill-no-p^2-wreathing}
Let $R$ be an elementary abelian group of rank 4. Under the notation we have established, suppose there are blocks $C, C' \in \mathcal B_2$ such that $G_{v,\mathcal B_2}$ fixes each $B\in \mathcal B_1$ with $B \subseteq C$, and also fixes each $B' \in \mathcal B_1$ with $B' \subseteq C'$. Further suppose that  $\alpha\in R_L$ is such that $\alpha(C_v)=C$ and there is no $i$ such that $\alpha^i(C_v)=C'$. Then there exists $\psi \in G^{(2)}$ that commutes with $\tau_0$ and $\tau_1$, such that $\tau_2 \in \psi^{-1}\pi^{-1}R_L\pi\psi$.
\end{cor}

\begin{proof}
This is an immediate consequence of Lemmas~\ref{still-no-p^2-wreathing} and~\ref{nop^2-wreathing}.
\end{proof}

This gives us the following conclusion.

\begin{cor}\label{done-if-no-p^2-wreathing}
Let $R$ be an elementary abelian group of rank 4. Under the notation we have established, if either:
\begin{itemize}
\item $\tau_2^{-1}\tau_2'$ fixes every block of some system of imprimitivity of $R_L$ with blocks of size $p$ that is a refinement of $\mathcal B_2$; or
\item there are blocks $C, C' \in \mathcal B_2$ such that $G_{v,\mathcal B_2}$ fixes each $B\in \mathcal B_1$ with $B \subseteq C$, and also fixes each $B' \in \mathcal B_1$ with $B' \subseteq C'$. Furthermore, if $\alpha\in R_L$ is such that $\alpha(C_v)=C$ then there is no $i$ such that $\alpha^i(C_v)=C'$,
\end{itemize}
then there exists $\psi \in G^{(2)}$ such that $\psi^{-1}\pi^{-1}R_L \pi\psi=R_L$.
\end{cor}

\begin{proof}
This follows from Corollary~\ref{stillstill-no-p^2-wreathing}, Lemma~\ref{nop^2-wreathing}, and Lemma~\ref{tau1}.
\end{proof}

Thus, in the next section, we may assume that there is no system of imprimitivity of $R_L$ with blocks of size $p$ that is a refinement of $\mathcal B_2$ whose blocks are all fixed by $\tau_2^{-1}\tau_2'$. We may also assume that if $\alpha, \beta\in R_L$ such that there is no $i$ for which $\alpha^i(C_v)=\beta(C_v)$, then there is no system of imprimitivity of $R_L$ with blocks of size $p$ that is a refinement of $\mathcal B_2$ such that the orbits of $G_{v, \mathcal B_2}$ in both $\alpha(C_v)$ and $\beta(C_v)$ are subsets of these blocks.

\section{Proof of Theorem 1.1}\label{therest} 


We begin with a lemma that gives us important information about certain orbits of $G_v$. In the lemma, we show that if $v \in D_v\in \mathcal B_i$, and $\alpha\in R_L$, then knowing something about the orbits of $G'$ in $\alpha(D_v)$ informs us about the orbits of $G'$ in $\alpha^j(D_v)$ for any $j$. Specifically, if the orbits in $\alpha(D_v)$ are not contained within the blocks of some smaller block system $\mathcal B_k$ of $R_L$, then neither are the orbits in $\alpha^j(D_v)$.

\begin{lemma}\label{orbits-not-smaller}
Let $R$ be an elementary abelian group. Under the notation we have established,
suppose $D, D_v \in \mathcal B_i$, with $v \in D_v$ and $D=\alpha(D_v)$ for some $\alpha \in R_L$.  
Suppose that some element of $G_v$ does not fix some block of $\mathcal B_j$ in $D$ (setwise), where $j <i$. Then for any $1\le k \le p-1$, some element of $G_v$ does not fix some block of $\mathcal B_j$ in $\alpha^k(D_v)$ (setwise).
\end{lemma}

\begin{proof}
Let $t$ be as small as possible such that some element of $G_v$ does not fix some block of $\mathcal B_j$ in $\alpha^{kt}(D_v)$ (setwise). Such a $t$ exists since $\alpha$ has order $p$, so $\alpha^{kt}=\alpha$ for some $1 \le t \le p-1$. By assumption, there exist $F\in \mathcal B_j$ such that $F\subset \alpha^{kt}(D_v)$ and $g \in G_v$ such that $g(F) \neq F$.  Let $\beta\in R_L$ be such that $\beta^{kt}(v) \in F$.

Now, by our choice of $t$, every element of $G_v$ fixes every block of $\mathcal B_j$ in $\alpha^{k(t-1)}(D_v)$ (setwise). In particular, if $F_v$ is the block of $\mathcal B_j$ that contains $v$, then $\beta^{k(t-1)}(F_v)$ is fixed (setwise) by $g$, so there exists some $\gamma \in R_L$ such that $\gamma g$ fixes the point $\beta^{k(t-1)}(v)$. Conjugating $\gamma g$ by $\beta^{k(t-1)}$ yields an element of $G_v$ that does not fix $\beta^k(F_v)\subset \beta^k(D_v)=\alpha^k(D_v)$.
\end{proof}

%
We must deepen our understanding of the orbits of $G_v$. We define a new relation $\sim$ on the points of $R$ as follows. We say $v_1 \sim v_2$ if there exists $v_3 \in C_{v_2}$ such that there is no $g \in G_{v_1}$ with $g(v_2)=v_3$, i.e. $C_{v_2}$ is not contained in an orbit of $G_{v_1}$. Notice that Lemma~\ref{orbits-not-smaller} shows that this relation is symmetric, since if $v_2=\alpha(v_1)$ then there is some $i$ such that $\alpha^i(v_2) \in C_{v_1}$, and the lengths of the intersection of the orbits of $G_{v_2}$ in $C_{v_1}=\alpha^i(C_{v_2})$ are the same as the lengths of the orbits of $G_{v_1}$ in $\alpha^i(C_{v_1})$, since these are conjugate. The relation $\sim$ need not be transitive, but we can define an equivalence relation $\equiv_2$ by $v_1 \equiv_2 v_2$ if there is a sequence $u_1=v_1, u_2, \ldots, u_i=v_2$ such that $u_1\sim u_2\sim \ldots \sim u_i$.

In the case $n=4$, the relation $\equiv_2$ may have $1$, $p$, or $p^2$ equivalence classes, since clearly any two vertices in the same block of $\mathcal B_2$ are equivalent. If $\equiv_2$ has more than one equivalence class, then each equivalence class has the form $\cup_{i=0}^{p-1} \alpha^i(C_w)$ for some $\alpha\in R_L$ (since the equivalence classes are blocks of $G$; note that if there are $p^2$ equivalence classes, then $\alpha$ fixes $C_w$). Thus Lemma~\ref{orbits-not-smaller} in fact proves that $\sim$ is an equivalence relation in this case.

Before proving our main lemmas, we prove a result that will be needed in both.

\begin{lemma}\label{tau_2(C_v)}
Let $R$ be an elementary abelian group of rank 4. Under the notation we have established, $v \not\sim w$ for any $w \in \tau_2(C_v)$.
\end{lemma}

\begin{proof}
By Corollary~\ref{done-if-no-p^2-wreathing}, we may assume that for any $\mathcal B_1$, some block of $\mathcal B_1$ is not fixed by $\tau_2^{-1}\tau_2'$. Let $\mathcal B_1$ be an arbitrary refinement of $\mathcal B_2$ with blocks of size $p$.

By the definition of $\sim$, showing that $v\not\sim w$ is equivalent to showing that $\tau_2(C_v)$ is contained in an orbit of $G_{v}$. Let $\alpha(B_v)$ be a block of $\mathcal B_1$ that is not fixed by $\tau_2^{-1}\tau_2'$. Consider $(\alpha')^{-1}\alpha$, where $\alpha'$ is the element of $\pi^{-1}R_L \pi$ such that $\alpha'(v)=\alpha(v)$. Clearly $(\alpha')^{-1}\alpha \in G_v$. Notice that $$\alpha'(\tau_2(B_v))=\alpha'(\tau_2'(B_v))=\tau_2'(\alpha'(B_v))=\tau_2'(\alpha(B_v))\neq \tau_2(\alpha(B_v)).$$ Therefore, $(\alpha')^{-1}\alpha(\tau_2(B_v))=(\alpha')^{-1}\tau_2\alpha(B_v)\neq \tau_2(B_v)$. Hence the orbits of $G_v$ in $\tau_2(C_v)$ are not contained in the blocks of $\mathcal B_1$.  There was nothing special about the choice of $\mathcal B_1$, so the orbits of $G_v$ in $\tau_2(C_v)$ are not contained in the blocks of any system of imprimitivity of $G$ that is a refinement of $\mathcal B_2$ with blocks of size $p$. The only way this can happen is if $\tau_2(C_v)$ is contained in an orbit of $G_v$, as claimed.
\end{proof}

In the next result, we dispose of the cases where $\equiv_2$ (and so $\sim$) have more than one equivalence class.

\begin{lemma}\label{wreathing}
Let $R$ be an elementary abelian group of rank 4. Under the notation we have established, suppose that $\equiv_2$ has more than one equivalence class. Then there exists $\psi \in G^{(2)}$ such that $\psi^{-1}\pi^{-1}R_L \pi\psi=R_L$.
\end{lemma}

\begin{proof}
By Lemma~\ref{tau_2(C_v)}, we know that $v \not\sim \tau_2(v)$. Since $\sim$ is actually an equivalence relation under our current assumptions, this in fact implies that 
 $\tau_2(v)$ is not in the same equivalence class (of $\equiv_2$) as $v$. 
 
Redefine $\tau_3$ if necessary, so that if $\equiv_2$ has $p$ equivalence classes, then the equivalence class containing $C_v$ is $\cup_{i=0}^{p-1}\tau_3^i(C_v)$. Define $\varphi$ by $\varphi(w)=\tau_2'^{j}\tau_2^{-j}(w)$ for $w \in \tau_2^{j}\tau_3^i(C_v)$.

First notice that $\varphi$ commutes with $\tau_0$ and $\tau_1$. By Corollary~\ref{tau3'}, we may therefore assume that $\tau_0,\tau_1 \in \varphi^{-1}\pi^{-1}R_L\pi\varphi$. Let $w \in R$ be arbitrary, say $w \in \tau_2^j\tau_3^i(C_v)$. Then we have $$\varphi^{-1}\tau_2'\varphi(w)=\varphi^{-1}\tau_2'\tau_2'^j\tau_2^{-j}(w)=\tau_2^{j+1}\tau_2'^{-(j+1)}\tau_2'^{j+1}\tau_2^{-j}(w)=\tau_2(w),$$ so $\varphi^{-1}\tau_2'\varphi=\tau_2$.  Notice that since $\tau_2^{-1}\tau_2'$ is in $G_{v,\mathcal B_2}$, $\varphi$ also fixes every block of $\mathcal B_2$ (setwise). By the definition of $\equiv_2$, in order to ensure that $\varphi \in G^{(2)}$, we need only verify that for any pair $w_1, w_2$ with $w_1 \equiv_2 w_2$, there is some $g \in G$ such that $g(w_1)=\varphi(w_1)$ and $g(w_2)=\varphi(w_2)$. But this is clear from the definition of $\varphi$, with $g=\tau_2'^j\tau_2^{-j}$ since $w_1\equiv_2 w_2$ implies that if $w_1 \in \tau_2^j\tau_3^{i_1}(C_v)$, then $w_2 \in \tau_2^j\tau_3^{i_2}(C_v)$.

Thus, $\tau_0,\tau_1,\tau_2\in \varphi^{-1}\pi^{-1}R_L\pi\varphi$. Now Lemma~\ref{tau1} completes the proof. 
\end{proof}

We now complete the proof with a longer result that deals with the case where $\equiv_2$ has a single equivalence class.

\begin{lemma}\label{lack-of-wreathing}
Let $R$ be an elementary abelian group of rank 4. Under the notation we have established, suppose that $\equiv_2$ has a single equivalence class. Then there exists $\psi \in G^{(2)}$ such that $\psi^{-1}\pi^{-1}R_L \pi\psi=R_L$.\end{lemma}

\begin{proof}
Since $\equiv_2$ has a single equivalence class, there must exist $\alpha, \beta \in R_L$ such that $\alpha$ and $\beta$ each have order $p$ on the blocks of $\mathcal B_2$, there is no $i$ such that $\alpha^i(C_v)=\beta(C_v)$, and $v \sim\alpha(v), \beta(v)$. Notice that since we may assume (by Corollary~\ref{tau3'}) that $\tau_0,\tau_1 \in Z(G)$, the intersection of any orbit of $G_v$ with any block of $\mathcal B_2$ must be a block admitted by $R_L$.  Observe that since $v \sim \alpha(v),\beta(v)$ and $G_{v,\mathcal B_2} \le G_v$, using the second condition of Corollary~\ref{done-if-no-p^2-wreathing} we may assume that if the intersections of the orbits of $G_v$ with $\alpha(C_v)$ are blocks of $\mathcal B_1$, then the intersections of the orbits of $G_v$ with $\beta(C_v)$ are blocks of $\mathcal B_1'$, where $\mathcal B_1'\neq \mathcal B_1$.

By Lemma~\ref{tau_2(C_v)}, we know that $\alpha(C_v),\beta(C_v)\neq \tau_2^i(C_v)$ for any $i$, so there exist some $j,k$ such that $\alpha^j\tau_2^k(C_v)=\beta(C_v)$. Since $G_v$ fixes every block of $\mathcal B_3$ setwise, in particular it fixes the block containing $\alpha^j(C_v)$ and $\beta(C_v)$ setwise. Since $\mathcal B_1' \neq \mathcal B_1$ and $\tau_0,\tau_1 \in Z(G)$, there must not exist $g \in G_v$ such that $g\alpha^j(C_v)=\beta(C_v)$. So $G_v$ must fix every block of $\mathcal B_2$ in the block of $\mathcal B_3$ that contains $\alpha^j(C_v)$ and $\beta(C_v)$. Using Lemma~\ref{orbits-not-smaller} with $i=3$ and $j=2$, this implies that $G_v=G_{v,\mathcal B_2}$.

Without loss of generality, assume $\beta(C_v)=\tau_2^i\alpha(C_v)$.

For the remainder of the proof, we must consider two cases separately: either there is some $\gamma \in R_L$ such that $\gamma(C_v)=\tau_2^j\alpha(C_v)$ for some $j \neq 0,i$, and $v \sim \gamma(v)$; or $v \not\sim \tau_2^j\alpha(v)$ for every $j \neq 0,i$.

{\bf \mathversion{bold}Case 1. There is some $\gamma \in R_L$ such that $\gamma(C_v)=\tau_2^j\alpha(C_v)$ for some $j \neq 0,i$, and $v \sim \gamma(v)$.}

Again using Corollary~\ref{done-if-no-p^2-wreathing} we may assume that the orbits of $G_v$ in $\gamma(C_v)$ are blocks of $\mathcal B_1''$, where $\mathcal B_1''\neq \mathcal B_1',\mathcal B_1$.

In this case, we claim that for any $g \in G_v$, if $g(\alpha(v))=\sigma_1(\alpha(v))$ and $g(\beta(v))=\sigma_2(\beta(v))$, then for any $w=\alpha^a\beta^b(v)$, we have $g(w)=\sigma_1^a\sigma_2^b(w)$. Since $g$ commutes with $\tau_0$ and $\tau_1$, and every block of $\mathcal B_2$ can be written as $\alpha^a\beta^b(C_v)$ for some $a,b$, this completely determines the action of $g$.

To prove our claim, we first note that $\alpha$, $\beta$, and $\gamma$ are interchangeable in the arguments we will make (and in fact in the statement of the claim). Notice that $\alpha(C_v)=\beta^{k_1}\gamma^{k_2}(C_v)$ for some $k_1,k_2$; in fact (by multiplying by appropriate powers of $\tau_0$ and $\tau_1$) we can choose $\beta$ and $\gamma$ such that $\beta^{k_1}\gamma^{k_2}(v)=\alpha(v)$.

Using Lemma~\ref{orbits-not-smaller}, we have $g(\beta^{k_1}(B_v'))=\beta^{k_1}(B_v')$, so there is some $\sigma_2 \in R_L$ that fixes all blocks of $\mathcal B_1'$ such that $\sigma_2^{-1}g$ fixes $\beta^{k_1}(v)$. Conjugating $G_v$ by $\beta^{k_1}$, we know that $G_{\beta^{k_1}(v)}$ fixes $\beta^{k_1}\gamma^{k_2}(B_v'')$, so $\sigma_2^{-1}g\alpha(B_v'')=\alpha(B_v'')$. Thus $g\alpha(B_v'')=\sigma_2\alpha(B_v'')$. Similarly, there is some $\sigma_1 \in R_L$ that fixes all blocks of $\mathcal B_1$ such that for some $x$, $\sigma_1\sigma_2^{-x}g$ fixes $\gamma^{k_2}(v)$, and a similar argument shows that $g\alpha(B_v')=\sigma_1^{-1}\sigma_2^x\alpha(B_v')$, where $\sigma_1\sigma_2^{-x}$ fixes every block of $\mathcal B_1''$.  In particular, $g\alpha(v) \in \sigma_2\alpha(B_v'')\cap \sigma_1^{-1}\sigma_2^x\alpha(B_v')\cap \alpha(B_v)$, and this intersection consists of the single point $\sigma_1\alpha(v)$ (so $x=1$). Hence $g\alpha(v)=\sigma_1\alpha(v)$. This argument in fact shows that if we know $g\alpha(v)$, we can determine from it $g\beta^{k_1}(v)$ and $g\gamma^{k_2}(v)$.

Now consider $g(\alpha\beta^{k_1}(v))$. Since $\sigma_1^{-1}g$ fixes $\alpha(v)$, it lies in $\alpha G_v \alpha^{-1}$, so must also fix $\alpha\beta^{k_1}(B_v')$. Similarly, $\sigma_2^{-1}g$ fixes $\beta^{k_1}(v)$, so fixes $\alpha\beta^{k_1}(B_v)$. Hence $g\alpha\beta^{k_1}(v) \in \sigma_1\alpha\beta^{k_1}(B_v') \cap \sigma_2 \alpha\beta^{k_1}(B_v)$. Straightforward calculations show that the unique vertex in this intersection is $\sigma_1\sigma_2\alpha\beta^{k_1}(v)$. Now $g\alpha^2(v) \in \alpha^2(B_v)$, say $\sigma_1^x g\alpha^2(v)=\alpha^2(v)$. Then $\sigma_1^x g \alpha\beta^{k_1}(v) \in \alpha\beta^{k_1}(B_v'')$, since $\alpha\beta^{k_1}(C_v)=\gamma^{-k_2}\alpha^2(C_v).$ The unique solution to this is $x=-2$, so $g\alpha^2(v)=\sigma_1^{2}\alpha^2(v)$. By repeating this argument $p-3$ more times, we may conclude that $g\alpha^a(v)=\sigma_1^a\alpha^a(v)$, for any $a$. As previously mentioned, there is nothing special about $\alpha$ as compared to $\beta$ or $\gamma$, so a corresponding result for $\beta$ also holds. 

Finally, since any block of $\mathcal B_2$ can be written uniquely as $\alpha^a\beta^b(C_v)$, by showing that if $g\alpha^a(v)=\sigma_1^a\alpha^a(v)$ and $g\beta^b(v)=\sigma_2^b
\beta^b(v)$ (replacing $\sigma_2$ by an appropriate power), then $g\alpha^a
\beta^b(v)=\sigma_1^a\sigma_2^b(\alpha^a\beta^b(v))$ we will complete the proof of our 
claim. Again, our knowledge of the orbits of $G_v$ tells us that since $\sigma_1^{-a}g$ fixes 
$\alpha^a(v)$, it also fixes $\alpha^a\beta^b(B_v')$, and since $\sigma_2^{-b}g$ fixes $\beta^{b}(v)$, it 
also fixes $\alpha^a\beta^b(B_v)$. Thus $g(\alpha^a\beta^b(v)) \in \sigma_1^a(\alpha^a
\beta^b(B_v')) \cap \sigma_2^b(\alpha^a\beta^b(B_v))$. The unique point of intersection shows that 
$g(\alpha^a\beta^b(v))=\sigma_1^a\sigma_2^b(\alpha^a\beta^b(v))$. This completes the 
proof of our claim. 

If $\tau_2^{-1}\tau_2'$ fixes either $\alpha(v)$ or $\beta(v)$, then the above claim demonstrates that $\tau_2^{-1}\tau_2'$ in fact fixes every block of either $\mathcal B_1'$ or (respectively) every block of $\mathcal B_1$, and hence by Corollary~\ref{done-if-no-p^2-wreathing}, we are done. So we may assume that $\tau_2^{-1}\tau_2'(\alpha(v))=\sigma_1 \alpha(v)$ where $\sigma_1 \in R_L$ is not the identity, and $\tau_2^{-1}\tau_2'(\beta(v))\neq \beta(v)$. 

Now, we have $\alpha^{-1}\alpha' \in G_v$, so by our claim, if $\alpha^{-1}\alpha'$ fixes either $\alpha(v)$ or $\beta(v)$, then it must fix every block system of either $\mathcal B_1'$ or $\mathcal B_1$ (respectively). But then since the only point that is in both $\gamma(B_v'')$ and either one of $\gamma(B_v)$ or $\gamma(B_v')$ is $\gamma(v)$, we must have $\alpha^{-1}\alpha'$ also fixes $\gamma(v)$. Now using whichever two of $\alpha(v)$, $\beta(v)$, and $\gamma(v)$ are fixed here, to take the roles of $\alpha(v)$ and $\beta(v)$ in the proof of our claim, we see that $\alpha=\alpha'\in Z(G)$. But then $\tau_2'\alpha(v)=\alpha\tau_2'(v)=\alpha\tau_2(v)=\tau_2\alpha(v),$ which cannot be true since $\tau_2^{-1}\tau_2'$ does not fix $\alpha(v)$. Thus $\alpha^{-1}\alpha'$ cannot fix either $\alpha(v)$ or $\beta(v)$, and hence (by our claim) does not fix any vertex outside of $C_v$. 

In particular, if $\alpha^{-1}\alpha'(\alpha(v))=\mu_1(\alpha(v))$ where $\mu_1 \in R_L$ fixes every block of $\mathcal B_1$, and $\alpha^{-1}\alpha'(\beta(v))=\mu_2(\beta(v))$ where $\mu_2 \in R_L$ fixes every block of $\mathcal B_1'$, then $\alpha^{-1}\alpha'(\tau_2(v))=\alpha^{-1}\alpha'(\alpha^{-1}\beta)^k(v)$ for some $k$ such that $ki=1$, and by our claim this will be $\mu_1^{-k}\mu_2^k(\alpha^{-1}\beta)^k(v)=\mu_1^{-k}\mu_2^k\tau_2(v)$, with neither $\mu_1$ nor $\mu_2$ being the identity.  Also, since $\alpha'$ commutes with $\tau_2'$, we have $$\alpha'\tau_2(v)=\alpha'\tau_2'(v)=\tau_2'\alpha'(v)=\tau_2'\alpha(v)=\sigma_1\tau_2\alpha(v),$$ so $\alpha^{-1}\alpha'\tau_2(v)=\sigma_1\tau_2(v) \in \tau_2(B_v)$. Since $\mu_2$ moves the blocks of $\mathcal B_1$ in a $p$-cycle and $k\neq 0$, $\mu_1^{-k}\mu_2^k\tau_2(v) \not\in \tau_2(B_v)$, which is a contradiction.

All of this serves to show that this case cannot arise unless Corollary~\ref{done-if-no-p^2-wreathing} applies.

{\bf \mathversion{bold}Case 2. $v \not\sim \tau_2^j\alpha(v)$ for every $j \neq 0,i$.}

We need to set up some notation for this case. Let $\sigma_{a,b}$ be the element of $R_L$ that takes $\alpha^a\tau_2^b(v)$ to $\tau_2^{-1}\tau_2'\alpha^a\tau_2^b(v)$. Notice that $\sigma_{a,b}$ fixes every block of $\mathcal B_2$. For any $m$, define $k_m$ and $k_m'$ to be the unique values such that 
$$\sigma_{1,0}^{k_m}\sigma_{m,0}\sigma_{m,1}\ldots\sigma_{m,mi-1}=\sigma_{1,i}^{k_{m}'}.$$

Define $\varphi$ by $\varphi(\alpha^a(v))=\sigma_{1,0}^{k_a}\alpha^a(v)$ for any $a$, $\varphi(\tau_2^t\alpha^a(v))=(\tau_2')^t\varphi(\alpha^a(v))$, and $\varphi$ commutes with $\tau_0$ and $\tau_1$. We claim that $\varphi\in G^{(2)}$, and that $\varphi^{-1}\tau_2'\varphi=\tau_2$. With Lemma~\ref{tau1}, this will complete the proof.

We begin by showing that $\varphi \in G^{(2)}$.
Since $\varphi$ fixes every block of $\mathcal B_2$, according to the assumptions of this case, we need only verify two things: that if $w,x$ are such that $x \in \alpha^a(C_w)$ for some $a$, then there is some $g\in G$ such that $g(B_x)=\varphi(B_x)$ and $g(B_w)=\varphi(B_w)$; and that if $w,x$ are such that $x \in \beta^a(C_w)$ for some $a$, then there is some $g\in G$ such that $g(B_x')=\varphi(B_x')$ and $g(B_w')=\varphi(B_w')$

Suppose that $x \in \alpha^a(C_w)$, say $w \in \tau_2^t\alpha^s(C_v)$ and $x \in \tau_2^t \alpha^{s+a}(C_v)$. Then $\varphi(w)=(\tau_2')^t\sigma_{1,0}^{k_s}\tau_2^{-t}(w) \in (\tau_2')^{t}\tau_2^{-t}(B_w)$. Also, $\varphi(x)=(\tau_2')^t\sigma_{1,0}^{k_{s+a}}\tau_2^{-t}(x) \in (\tau_2')^{t}\tau_2^{-t}(B_x)$, giving us the desired conclusion with $g=(\tau_2')^t\tau_2^{-t}$.

Suppose that $x \in \beta^a(C_w)$,  say $w \in \tau_2^t\alpha^s(C_v)$ and $x \in \tau_2^{t+ia} \alpha^{s+a}(C_v)$. Then 
\begin{eqnarray*}
\varphi(w)&=&(\tau_2')^t\sigma_{1,0}^{k_s} \tau_2^{-t}(w)\\
&=&(\tau_2')^t\sigma_{1,i}^{k_s'}\sigma_{s,0}^{-1}\sigma_{s,1}^{-1}\ldots\sigma_{s,si-1}^{-1} \tau_2^{-t}(w)\\
&=&\sigma_{1,i}^{k_s'}\sigma_{s,0}^{-1}\sigma_{s,1}^{-1}\ldots\sigma_{s,si-1}^{-1}\sigma_{s,0}\sigma_{s,1}\ldots\sigma_{s,t-1}(w)\\
&\in&\sigma_{s,0}^{-1}\sigma_{s,1}^{-1}\ldots\sigma_{s,si-1}^{-1}\sigma_{s,0}\sigma_{s,1}\ldots\sigma_{s,t-1}(B_w')\\
&=& (\tau_2')^{t-si}\tau_2^{si-t}(B_w').
\end{eqnarray*}
Similarly, 
\begin{eqnarray*}
\varphi(x)&=&(\tau_2')^{t+ia}\sigma_{1,0}^{k_{s+a}} \tau_2^{-t-ia}(x)\\
&=&(\tau_2')^{t+ia}\sigma_{1,i}^{k_{s+a}'}\sigma_{s+a,0}^{-1}\sigma_{s+a,1}^{-1}\ldots\sigma_{s+a,(s+a)i-1}^{-1} \tau_2^{-t-ia}(x)\\
&\in&\sigma_{s+a,0}^{-1}\sigma_{s+a,1}^{-1}\ldots\sigma_{s+a,(s+a)i-1}^{-1}\sigma_{s+a,0}\sigma_{s+a,1}\ldots\sigma_{s+a,t+ai-1}(B_x')\\
&=& (\tau_2')^{t-si}\tau_2^{si-t}(B_x').
\end{eqnarray*}
Again we have the desired conclusion, with $g=(\tau_2')^{t-si}\tau_2^{si-t}$. Hence $\varphi\in G^{(2)}$.

Now consider $\varphi^{-1}\tau_2'\varphi(w)$, where $w \in \tau_2^t\alpha^s(C_v)$. Although the calculation in the case $t=p-1$ is different from the other cases, straightforward calculations show that this is $\tau_2(w)$ in all cases.
\end{proof}

We have now proven our main theorem.

\begin{proof}[Proof of Theorem~\ref{thrm:main}] Lemmas~\ref{wreathing} and~\ref{lack-of-wreathing} together prove the desired result.
\end{proof}

\end{document}